\documentclass{amsart}

\usepackage[makeroom]{cancel}
\usepackage{graphicx}
\usepackage{cancel}
\usepackage{ulem}
\usepackage{setspace}

\usepackage{tikz}
\usepackage{verbatim}

\usepackage{amsrefs}

\usepackage{graphicx}
\usepackage{mathtools}

\usepackage{graphicx}
\usepackage{float} 
\usepackage{amsfonts}
\usepackage{amsthm}
\usepackage{latexsym}
\usepackage{amsmath}
\usepackage{amssymb}
\usepackage{amscd}
\usepackage{epsf}
\usepackage{tikz}
\usetikzlibrary{matrix,arrows}
\usepackage{enumerate}
\usepackage{eepic}
\usepackage{epic}
\usepackage{amsrefs}

\usepackage{amsthm}
\newcommand*\Bell{\ensuremath{\boldsymbol\ell}}
\newcommand*\Bxi{\ensuremath{\boldsymbol\xi}}
\newcommand*\Btheta{\ensuremath{\boldsymbol\theta}}

\theoremstyle{definition}
\newtheorem{definition}{Definition}

\theoremstyle{plain}
\newtheorem{theorem}{Theorem}

\newtheorem{lemma}{Lemma}

\newtheorem*{theorem*}{Theorem}

\newtheorem*{definition*}{Definition}

\title{Stable polynomials and crystalline measures}

\author{Pavel Kurasov}

\author{Peter Sarnak}

 \date{\today}

\begin{document}

\maketitle

\begin{abstract}
Explicit examples of {\bf positive} crystalline measures and Fourier quasicrystals 
are constructed using pairs of stable of polynomials, answering several
open questions in the area.
\end{abstract}

\section{Introduction}

Our investigation  of 
the additive structure of the spectrum of metric graphs \cite{KuSa2} provides exotic crystalline measures, in fact ones that give
 answers to a number of open problems. In this note we explicate the simplest examples and place the construction into the natural general setting of stable polynomials in several variables. \vspace{3mm}

We recall the definitions. 

\begin{definition*}
A {\bf crystalline measure} $ \mu $ on $ \mathbb R $ is a tempered distribution of the form
\begin{equation} \label{mudef}
\mu = \sum_{\lambda \in \Lambda} a_\lambda \delta_\lambda, \quad  \mbox{for which} \quad \hat{\mu} = \sum_{s \in S} b_s \delta_s,
\end{equation}
where $ \delta_\xi $ is a delta mass at $ \xi , $ and $ \Lambda $ and $ S $ are discrete subsets of $ \mathbb R $ \cite{Me16}.

If both $ | \mu | $ and $ |\hat{\mu}| $ are tempered as well, then following 
\cite{LeOl17}*{Section 1.1}
we call $ \mu $ a  {\bf Fourier quasicrystal}.
\end{definition*}

The basic example of a crystalline measure, in fact a Fourier quasicrystal, comes from the Poisson summation formula:
\begin{equation} \label{Psf}
\mu = \sum_{m \in \mathbb Z} \delta_m  \Rightarrow  \hat{\mu} = \sum_{s \in \mathbb Z} \delta_{2 \pi s},
\end{equation}
and its extension to finite combinations of these called ``generalized Dirac combs'' \cite{Me16}.  Various examples of
crystallline measures that are not  Dirac combs were constructed by Guinand \cite{Gu59}.
Note however that his example 4 page 264 coming from the explicit formula in the theory of primes does not give a Fourier quasicrystal, even assuming the Riemann hypothesis.

Towards a classification theory of crystalline measures $ \mu $ there 
are a series of results that ensure that $ \mu $ is a generalized Dirac comb (\cite{Me70,LeOl15,LeOl17,Co88}), one
of the first being

\begin{theorem*}[Meyer \cite{Me70}]   \label{ThMe}
If $ a_\lambda $ take values in a finite set and $ |\hat{\mu}| $ is translation bounded, that is $ \displaystyle  \sup_{x \in \mathbb R} |\hat{\mu}| (x+ [0,1]) < \infty $, then $ \mu $ is a generalized Dirac comb.
\end{theorem*}

Examples of varying complexity of Fourier quasicrystals which are not generalized Dirac combs, have been given (\cite{Me16,LeOl16,Ko16,RaVi19}), showing that any such classification is probably very difficult \cite{Dy09}.

A basic question which has been open for some time is whether there are positive (that is with $ a_\lambda \geq 0 $) crystalline measures which are not generalized Dirac combs?
The constructions in Sections 2 and 3 yield such
$ \mu$'s which enjoy some other properties which resolve related open problems.

In Section 2 we review the definition of stable polynomials and use them to construct positive Fourier quasicrystals.
In Section 3 we examine the simplest non-trivial example and use Liardet's proof of Lang's conjecture in dimension two \cite{La65,Li74}
to analyze the additive structure of $ \Lambda $, see Theorem \ref{ThLiardet}. This example is rich
enough for the purposes of this note.
We end the section by recording the general additive structure theorem from \cite{KuSa2} which applies to the supports
$ \Lambda $ of the Fourier quasicrystal measures $ \mu $ that are
constructed from stable polynomials.

\section{Summation formula}

\subsection*{Stable polynomials}

If $ P(\mathbf z) = P (z_1, \dots, z_n ) $ is a multivariable polynomial with complex coefficients, we say that
$ P $ is $ \mathbb D = \left\{  z: |z|<1 \right\} $ stable if $ P(\mathbf z) \neq 0 $ for
$ {\bf z} = (z_1, z_2, \dots, z_n ) $ with $ z_j \in \mathbb D $ for all $ j. $
To define a stable pair, consider the involution operation on $ P $
obtained by $ z_j \rightarrow 1/z_j $ for $ j = 1,2, \dots, n $, the result being denoted by $ P^\iota . $

\begin{definition*}
Two multivariate polynomials $ P, Q $ are said to form {\bf stable pair} if 
\begin{enumerate}
\item both polynomials $ P $ and $ Q $ are $ \mathbb D$-stable;
\item there exist an integer-valued vector $ \Bell = (\ell_1, \ell_2, \dots, \ell_n) \in \mathbb N^n $
and a constant $ \eta $ such that $ P $ and $ Q $ satisfy the functional equation
\begin{equation} \label{defQ}
Q(\mathbf z) = \eta \; z_1^{\ell_1} z_2^{\ell_2} \dots z_n^{\ell_n} P^\iota (\mathbf z);
\end{equation}
\item the normalization condition
$$ P(\vec{0})  = Q (\vec{0}) = 1 $$
is fulfilled.
\end{enumerate}
\end{definition*}

If such $ \Bell $ and $ \eta $ exists they are unique. \vspace{3mm}

Such stable pairs arise in many contexts and there are powerful techniques for proving stability \cite{BoBr09,Wa11}.
We point to two basic examples.

{\bf 1) } {\bf Spectral pairs} These come up as secular polynomials in quantum graphs \cite{BaGa,CdV}.
Let $ P_1, P_2, \dots, P_k $ be monomials in $ z_1, \dots, z_n $ of the form
\begin{equation}
P_j (\mathbf z) = z_1^{a_{j,1}} z_2^{a_{j,2}} \dots z_n^{a_{j,n}}, \quad j = 1,2, \dots, k,
\end{equation}
$ a_{j, \nu } \in \mathbb N. $
Let $ \ell_\nu = \sum_{j=1}^k a_{j,\nu}, $ which we assume being positive for every $ \nu = 1, \dots, n. $
If $ S $ is a $ k \times k $ unitary matrix, set
\begin{equation}
R_S (\mathbf z) =
\det \left[ \mathbb I_{k \times k} -
\left[
\begin{array}{cccc}
P_1(\mathbf z) & 0 & \dots & 0 \\
0 & P_2(\mathbf z) & \dots & 0 \\
\vdots & \vdots & \ddots & \vdots \\
0 & 0 & \dots & P_k (\mathbf z)
\end{array} \right] S \right].
\end{equation}
Then it is easy to see that
$$ P = R_S, \quad Q = R_{S^{-1}}, $$
is a stable pair with $  \Bell = (\ell_1, \dots, \ell_k) $ and $ \eta = (\det (-S) )^{-1}. $

 Our studies in \cite{KuSa2} were inspired by the trace formula for metric graphs
\cite{GuSm,KuJFA,KuArk}.

{\bf 2)} {\bf Lee-Yang pairs } (\cite{Ru}*{Theorem 5.12})
Let  $ -1 \leq A_{ij} \leq 1, A_{ij} = A_{ji} $ and
\begin{equation}
P(\mathbf z) = \sum_{S} \left(
\prod_{i \in S} \prod_{j \in S'} A_{ij}  \right) \mathbf z^S,
\end{equation}
where  we use multi-index notation for $ \mathbf z^S = \prod_{j \in S} z_j $, the sum is over all subsets $ S $ of $ \{ 1,2, \dots, n\} $ and $ S' $ is the complement of $ S$.
Then $ P $ is a self dual stable pair
\begin{equation}
(z_1 z_2 \dots z_n)\;  P^\iota (\mathbf z) = P(\mathbf z).
\end{equation}
For generalizations of these see \cite{YaLe2,BoBr09,Wa11}.

For the rest of this section we show how to attach to a stable pair
and real numbers $ b_1, b_2, \dots, b_n > 1 $ a crystalline measure.

\subsection*{Notations}
Assume that $ P , Q $ is a stable pair of multivariable polynomials:
\begin{equation}
P (\mathbf z) = 1 + \sum_{\mathbf m \in M_P} a_P (\mathbf m) \mathbf z^\mathbf m, \quad
Q (\mathbf z) = 1 + \sum_{\mathbf m \in M_Q} a_Q (\mathbf m) \mathbf z^\mathbf m,
\end{equation} 
where $ M_{P,Q} $ are finite subsets of
\begin{equation}
\mathbf Z_+^n := \left\{ \mathbf k = (k_1, k_2, \dots, k_n), k_j \in \mathbb Z, k_j \geq 0, \mathbf k \neq (0,0, \dots,0) \right\}.
\end{equation}
Taking the logarithm we get the following expansion
\begin{equation}
\begin{array}{ccl}
\log P(\mathbf z) & = & \displaystyle
\sum_{\nu =1}^\infty \frac{(-1)^\nu}{\nu} \left( \sum_{\mathbf m \in M_P} a_P (\mathbf m) \mathbf z^\mathbf m \right)^\nu \\
& = & \displaystyle
\sum_{\nu =1}^\infty \frac{(-1)^\nu}{\nu}
\sum_{\mathbf m_1, \mathbf m_2, \dots, \mathbf m_\nu \in M_P}
a_P (\mathbf m_1) a_P (\mathbf m_2) \dots a_P (\mathbf m_\nu) \mathbf z^{\mathbf m_1}  \mathbf z^{\mathbf m_2} \dots
 \mathbf z^{\mathbf m_\nu} \\
 & = & \displaystyle \sum_{\nu =1}^\infty \frac{(-1)^\nu}{\nu}
 \sum_{\mathbf k \in \mathbb Z_+^n} \sum_{\mathbf m_1 + \mathbf m_2 + \dots + \mathbf m_\nu = \mathbf k }
a_P (\mathbf m_1) a_P (\mathbf m_2) \dots a_P (\mathbf m_\nu) \mathbf z^{\mathbf k},
  \end{array}
\end{equation}
hence
\begin{equation} \label{logexp}
\log P (\mathbf z) = \sum_{\mathbf k \in \mathbb Z_+^n} c_P (\mathbf k) \mathbf z^\mathbf k,
\end{equation}
where for $ k \in \mathbb Z_+^n $
\begin{equation}
c_P (\mathbf k) = \sum_{\nu=1}^{\sum_{j=1}^n k_j}  \sum_{\mathbf m_1 + \mathbf m_2 + \dots + \mathbf m_\nu = \mathbf k }
\frac{(-1)^\nu a_P (\mathbf m_1) a_P (\mathbf m_2) \dots a_P (\mathbf m_\nu)}{\nu}.
\end{equation}
Similar formulas hold for $ \log Q(\mathbf z). $

\subsection*{Dirichlet series}

Let $ b_1, b_2, \dots, b_n $ be real numbers larger than $ 1 $ and let $ \xi_j = \ln b_j >0, \; j =1,2, \dots, n. $

Let us denote by $ \Gamma_+ $ and $ L_+ $ the corresponding multiplicative and additive semigroups

\begin{equation} \label{defL}
\begin{array}{ccl}
\Gamma_+ & = & \displaystyle 
 \left\{ b_1^{m_1} b_2^{m_2} \dots b_n^{m_n}: m_j \in \mathbb N \cup \{0\}  \right\} \setminus \{1\}; \\[5mm]
 L_+ & = & \displaystyle
  \log \Gamma_+ = \left\{ m_1 \xi_1 + m_2 \xi_2 + \dots + m_n \xi_n: m_j \in \mathbb N \cup \{0\}  \right\} 
\setminus \{0\}.
\end{array}
\end{equation}
The elements of these semigroups will be denoted by $ \mathbf b $ and $ \Bxi $ respectively
$$ \mathbf b \in  \Gamma_+ , \quad \Bxi \in L_+ .$$

Let us introduce the following two entire functions of order $1$
\begin{equation} \label{defF}
\begin{array}{ccll}
F(s) &  :=  & P (b_1^{-s}, b_2^{-s}, \dots, b_n^{-s} ) \equiv P(e^{-\xi_1 s},e^{-\xi_2 s}. \dots, e^{-\xi_n s}) , & \quad s \in \mathbb C. \\[3mm]
G(s) & := &  Q (b_1^{-s}, b_2^{-s}, \dots, b_n^{-s} )\equiv Q(e^{-\xi_1 s},e^{-\xi_2 s}. \dots, e^{-\xi_n s}), 
\end{array}
\end{equation}
The functions are related via the functional equation
$$ \begin{array}{ccl}
F(-s) & = & P (b_1^{s}, b_2^s , \dots , b_n^s) \\
& = & \displaystyle
P^{\iota} (b_1^{-s}, \dots, b_n^{-s}) \\
& = & \displaystyle   \eta^{-1}  \left(b_1^{\ell_1} \dots b_n^{\ell_n} \right)^s 
G(s)
\end{array}
$$
\begin{equation} \label{F-G}
\Rightarrow F(-s) = \eta^{-1}  \left( \mathbf b_1^{\Bell} \right)^s G(s),
\end{equation}
where $ \Bell = (\ell_1, \ell_2, \dots, \ell_n). $

The stability conditions on $ P $ and $ Q $
ensure that all zeroes of $ F(s) $ and $ G(s) $  are on the imaginary axis $ \Re (s) = 0. $
Moreover \eqref{F-G} implies that the zeroes for $ F $ an $ G $ are obtained from each other via reflection.

$ F $  and $ G $ are finite Dirichlet series, that is
\begin{equation}
F(s) = 1+ \sum_{\mathbf m \in  M_{P}} a_P (\mathbf m)
 ({\mathbf b}^{\mathbf m})^{-s}, \quad G(s) = 1+ \sum_{\mathbf m \in  M_{Q}} a_Q (\mathbf m)
 ({\mathbf b}^{\mathbf m})^{-s}.
\end{equation}

\subsection*{Logarithmic derivatives}

For $ \Re (s) $ large enough the series for $ \log F(s) $
converges absolutely:
\begin{equation} 
\log F(s) =  \sum_{\mathbf k \in \mathbb Z_+^n } c_P(\mathbf k)  e^{-(k_1 \xi_1 + k_2 \xi_2 + \dots + k_n \xi_n)s}  = 
 \sum_{\mathbf k \in \mathbb Z_+^n } c_P(\mathbf k)  e^{- (\Bxi \cdot \mathbf k)s}.
\end{equation}
Hence for $ \Re (s) $ large
\begin{equation}  \label{expF} 
\frac{F'(s)}{F(s)} = -  \sum_{\mathbf k \in \mathbb Z_+^n } (\Bxi \cdot \mathbf k)
c_P(\mathbf k)  e^{-(\Bxi \cdot \mathbf k)s}.
\end{equation}

A similar analysis can be applied to the entire function $ G(s) $ 
leading to
\begin{equation}  \label{expG} 
\frac{G'(s)}{G(s)} =  -  \sum_{\mathbf k \in \mathbb Z_+^n }  (\Bxi \cdot \mathbf k)
c_Q(\mathbf k)  e^{-(\Bxi \cdot \mathbf k)s}.
\end{equation}

Formula \eqref{F-G} establishes the following relation between the logarithmic derivatives of $ F $ and $G $
$$
\log F(-s) =  - \log \eta + s (\Bxi \cdot \Bell )  + \log G(s) 
$$
\begin{equation} \label{F-G2}
\Rightarrow - \frac{F'(-s)}{F(-s)} =(\Bxi \cdot \Bell)  + \frac{G'(s)}{G(s)} .
\end{equation}
Note that this relation is independent of the parameter $ \eta $ appeared first in \eqref{defQ}.

\subsection*{Logarithmic derivative as a distribution}

Let $ \Psi \in C_0^\infty (\mathbb R_{>0} ) $ and
\begin{equation}
\tilde{\Psi} (s) = \int_0^\infty \Psi (x) x^s \frac{dx}{x}.
\end{equation}
$ \tilde{\Psi} (s) $ is entire and is rapidly decreasing when $ |t| \rightarrow \infty $ for $ s = \sigma + it, $ $ \sigma $ fixed.

Consider the integral
\begin{equation} \label{I}
I :=  \displaystyle
\frac{1}{2 \pi i } \int_{\Re (s) = R}
\frac{F'(s)}{F(s)} \tilde{\Psi} (s) ds,
\end{equation}
which is converging for large real $ R. $ We next calculate $ I $  in two different ways
using the functional equation connecting $ F $ and $G $.

Expansion \eqref{expF} gives us
\begin{equation}  \label{I1}
\begin{array}{ccl}
I & = &  \displaystyle \frac{1}{2 \pi i} \int_{\Re (s) = R}
\left( - \sum_{\mathbf k  \in \mathbb Z_+^n  }  (\Bxi \cdot \mathbf k) c_P (\mathbf k) e^{- (\Bxi \cdot \mathbf k) s} 
 \right) \tilde{\Psi} (s) ds  \\
& = & \displaystyle 
- \sum_{\mathbf k  \in \mathbb Z_+^n  }  (\Bxi \cdot \mathbf k) c_P(\mathbf k) 
 \frac{1}{2 \pi i }  \int_{\Re(s) = R }  \tilde{\Psi}(s) e^{- (\Bxi \cdot \mathbf k)s} ds.
\end{array}
\end{equation}

To get the second representation we shift the contour for the integral defining $ I $ to $ \Re (s) = - R $ picking up
the residues, which are $ \tilde{\Psi} (\rho) $, since the function $ \tilde{\Psi} $
is integrated with the logarithmic derivative. 
 Summing over all zeroes of $ F $ (which are lying on the imaginary axis)
we obtain
\begin{equation}
\sum_{\rho: F(\rho) = 0 } \tilde{\Psi} (\rho),
\end{equation}
hence
$$
\begin{array}{ccl}
I  & = & \displaystyle \sum_{\rho: F(\rho) = 0 } \tilde{\Psi} (\rho) + \frac{1}{2\pi i}
\int_{\Re (s) = -R} \frac{F'(s)}{F(s)} \tilde{\Psi} (s) ds  \\
& = & \displaystyle
\sum_{\rho: F(\rho) = 0 } \tilde{\Psi} (\rho) + \frac{1}{2\pi i}
\int_{\Re (s) = R} \frac{F'(-s)}{F(-s)} \tilde{\Psi} (-s) ds
\end{array}
$$

Formula \eqref{F-G2} together with expansion \eqref{expG} then imply
\begin{equation} \label{I2}
\begin{array}{ccl}
I & = & \displaystyle
\sum_{\rho: F(\rho) = 0 }   \tilde{\Psi} (\rho) -   (\Bxi \cdot \Bell)  \frac{1}{2\pi i}
 \int_{\Re(s) = R} \tilde{\Psi} (-s) ds  \\[3mm]
 & & \displaystyle + \sum_{\mathbf k  \in \mathbb Z_+^n }
 (\Bxi \cdot \mathbf k) c_Q (\mathbf k)   \frac{1}{2 \pi i}
\int_{\Re(s) = R} \tilde{\Psi}(-s) e^{- (\Bxi \cdot \mathbf k) s}  ds 
\end{array}
\end{equation}
Comparing two formulas for $ I $  (expressions \eqref{I1} and \eqref{I2}) we may calculate the sum over the zeroes of $ F $
\begin{equation} \label{sum1}
\begin{array}{ccl}
\displaystyle \sum_{\rho: F(\rho) = 0 }    \tilde{\Psi} (\rho) & = &
\displaystyle
 (\Bxi \cdot \Bell)  \frac{1}{2\pi i} \int_{\Re(s) = R} \tilde{\Psi} (s) ds \\
&& \displaystyle 
- \sum_{\mathbf k \in \mathbb Z_+^n }  (\Bxi \cdot \mathbf k)  c_P (\mathbf k)  \frac{1}{2\pi i}
 \int_{\Re(s) = R}  \tilde{\Psi}(s) e^{- (\Bxi \cdot \mathbf k) s} ds \\
& & \displaystyle
- \sum_{\mathbf  k \in \mathbb Z_+^n }  (\Bxi \cdot \mathbf k) c_Q(\mathbf k)  \frac{1}{2\pi i}
 \int_{\Re(s) = R}  \tilde{\Psi}(-s)  e^{- (\Bxi \cdot \mathbf k) s} ds
\end{array}
\end{equation}
%%%%%%%%%%%%%%%%%%%
\subsection*{Summation formula}
%%%%%%%%%%%%%%%%%%
We make change of variables:
$$ \begin{array}{l}
\displaystyle x = e^{t};  \\ (0,+\infty) \rightarrow (-\infty, + \infty), 
\end{array} $$
so that
$$ \Psi (e^t) = h(t) $$
for a certain $ h \in C_0^\infty (\mathbb R). $
We have in particular:
$$
\tilde{\Psi} (i \gamma) = \int_0^\infty \Psi (x) x^{i \gamma} \frac{dx}{x} = \left[
\begin{array}{c}
x = e^t \\
dx = e^t dt
\end{array} \right] = \int_{-\infty}^\infty h(t) e^{i \gamma t} dt = \hat{h} (\gamma), $$
where $ \hat{h} $ is the Fourier transform of $ h $
and
$$
\begin{array}{ccl}
\displaystyle \frac{1}{2 \pi i} \int_{\Re (s) = R} \tilde{\Psi} (s) e^{- (\Bxi \cdot \mathbf k) s}  ds  & = &
\displaystyle
\frac{1}{2 \pi i} \int_{\Re ( s) = R}  \left( \int_0^\infty \Psi(x) x^s \frac{dx}{x} \right) e^{- (\Bxi \cdot \mathbf k) s} ds  \\[3mm]
& = & \displaystyle
\frac{1}{2 \pi } \int_{- \infty}^{+ \infty} \left(
\int_{-\infty}^\infty h(t) e^{(R+is) t} dt \right) e^{- (\Bxi \cdot \mathbf k) R } e^{-i  (\Bxi \cdot \mathbf k) s } ds  \\[3mm]
& = & h ( \Bxi \cdot \mathbf k).
\end{array}  $$

Then formula \eqref{sum1} becomes the following summation formula
\begin{equation} \label{sum2}
\sum_{\gamma: F(i \gamma) = 0}  \hat{h} (\gamma) =
 (\Bxi \cdot \Bell) h(0)
- \sum_{\mathbf k \in \mathbb Z_+^n }
 (\Bxi \cdot \mathbf k) c_P (\mathbf k) h  (\Bxi \cdot \mathbf k) -
  \sum_{\mathbf k \in \mathbb Z_+^n }
 (\Bxi \cdot \mathbf k) c_Q (\mathbf k) h  (-\Bxi \cdot \mathbf k)
 \end{equation}
which is valid for any $ \hat{h} \in C_0^\infty (\mathbb R), $
and extends to all of $ \mathcal S (\mathbb R) $
as shown in the proof of Theorem \ref{Th1} below.

To be precise,  introducing the discrete support set
$$ \Lambda_P :=
\{ \gamma: F(i \gamma) = 0 \}  $$
obtained from the zero set of $ F $ (all lying on
the imaginary axis) we define the discrete measure associated 
with 
the left hand side of \eqref{sum2}
\begin{equation} \label{defmu}
\mu_P :=  \sum_{\gamma: F(i\gamma) = 0} \delta_\gamma \equiv \sum_{\lambda \in \Lambda_P} m(\lambda) \delta_\lambda,
\end{equation}
where $ m(\lambda) $ is the multiplicity of the corresponding zero.

 Then the spectrum $ S_P $ of $ \mu $ is a subset of
$$ L_+ \cup - L_+ \cup \{ 0 \} .$$
(with $ L_+ $ introduced in \eqref{defL}) and the Fourier transform of $ \mu $ 
can be written as
\begin{equation} \label{defmuhat}
\hat{\mu} =  (\Bxi \cdot \Bell)  
- \sum_{\mathbf k \in \mathbb Z_+^n }
 (\Bxi \cdot \mathbf k) c_P (\mathbf k)  \delta_{  \Bxi \cdot \mathbf k}
 -  \sum_{\mathbf k \in \mathbb Z_+^n }
 (\Bxi \cdot \mathbf k) c_Q (\mathbf k)  \delta_{ - \Bxi \cdot \mathbf k}.
\end{equation}

\begin{theorem} \label{Th1}
Given any pair $ P, Q $ of stable polynomials  satisfying assumptions (1) and (2) 
the measure $ \mu $ is a positive crystalline measure, in fact a Fourier
quasi-crystal and is an almost periodic measure.
\end{theorem}
\begin{proof}
The support of $ \mu $ is given by the zeroes $ i \gamma_j $  of the entire function $ F $
 in \eqref{defF} and hence the support $ \Lambda $ of $ \mu $ is discrete. The support $ S $
of $ \hat{\mu}$ is a subset of $ L_+ \cup - L_+ \cup \{ 0\} $ which is also discrete.
Since $ m(\lambda) \geq 1 $
and  $ \mu $ is positive, applying the summation formula 
to $ \phi (y) = \phi_0(x-y) $ with $ \phi_0 \geq 0, \phi_0\geq1 $ on $ [-1,1] $
and $ \hat{\phi}_0 $ having compact support in $ (- \epsilon_0, \epsilon_0) $
where $ (-\epsilon_0, \epsilon_0) \cap (L_+ \cup - L_+) $ is empty,  yields
$ \displaystyle \sum_{\lambda: x-1 \leq \lambda \leq x+1} m(\lambda) \ll
(\Bxi \cdot \Bell) \hat{\phi}_0(0), $ uniformly in $ x. $
That is $ \mu = |\mu| $ is translation bounded and in particular $ \mu $
and hence $ \hat{\mu} $ are both tempered. This shows that $ \mu $ is 
 a crystalline measure. To show that it is a Fourier quasicrystal
we need to show in addition that $ | \hat{\mu}| $ is tempered (since $ \mu = | \mu|$).
To this end we first bound the coefficients $ c_P (\mathbf k) $ in  \eqref{logexp}. 
The series in \eqref{logexp} converges absolutely and uniformly for $ \mathbf z $
in compact subsets of $ \mathbb D^n = \mathbb D \times \mathbb D \times \dots \times \mathbb D, $
and yields $ \log P(\mathbf z) = \ln | P(\mathbf z) | + i \arg P(\mathbf z)$ where the $ \arg $
is gotten by continuous variation along the path $ \{ s \mathbf z \}, 0 \leq s \leq 1. $
Since $ P (s \mathbf z) $, as a function of $ s $ is a polynomial in $ s $ of degree $ \deg P $, it follows that
\begin{equation} \label{A}
| \arg P(\mathbf z) | \leq \pi (\deg P). 
\end{equation}
Let $  \displaystyle K = \sup_{\mathbf z \in \mathbb D^n} | P(\mathbf z) | $, then
\begin{equation} \label{B}
\ln | P(\mathbf z)| \leq  \ln K, \quad \mbox{for $ z \in \mathbb D^n$}.
\end{equation}
 Introducing  the notation 
$$ e^{i \Btheta} = (e^{i \theta_1}, e^{i \theta_2}, \dots, e^{i \theta_n} ) $$
we have from \eqref{logexp} that for $ 0 \leq r < 1 $ 
\begin{equation} \label{*}
\int_{\mathbb T^n} \log P ( r e^{i \Btheta}) e^{- i \mathbf k \cdot \Btheta} d \Btheta =
r^{[\mathbf k|} c_P (\mathbf k).
\end{equation}
In particular for $ k = 0 $
\begin{equation} \label{D}
\int_{\mathbb T^n} \ln | P ( r e^{i \Btheta}) | d \Btheta = 0,
\end{equation}
since the constant term is absent in \eqref{logexp}.
According to \eqref{B} 
$$ \ln | P ( r e^{i \Btheta}) | - \ln K \leq 0 $$
and hence
\begin{equation} \label{C}
\begin{array}{ccl}
\displaystyle  \int_{\mathbb T^n} 
\left\vert \ln | P ( r e^{i \Btheta}) | \right\vert d \Btheta
& = &
\displaystyle \int_{\mathbb T^n} \Big\vert \ln | P (r e^{i \Btheta})| - \ln K + \ln K \Big\vert d \Btheta \\[5mm]
& \leq &
\displaystyle - \int_{\mathbb T^n} \Big( \ln | P( r e^{i \Btheta}) | - \ln K \Big) d \Btheta + \ln K \\[5mm]
& = & 2 \ln K
\end{array}
\end{equation}
by \eqref{D}. From \eqref{*} and the $ r $ independent bounds \eqref{A} and \eqref{C} we deduce
\begin{equation}
\left\vert  c_P (\mathbf k) \right\vert \leq C < \infty.
\end{equation}
From \eqref{defmuhat}, it follows that the measure $ | \hat{\mu} | $ satisfies
\begin{equation}
\begin{array}{ccl}
| \hat{\mu} | ([-A,A])  & \leq &  \displaystyle \Bxi \cdot \Bell + 2 \sum_{\mathbf k \in \mathbb Z_+^n, \Bxi \cdot \mathbf k \leq A}
(\Bxi \cdot \mathbf k)  | c_P (\mathbf k) |  \\
& \leq &  \displaystyle \Bxi \cdot \Bell + 2 \max \{ \xi_j \} C \sum_{\mathbf k \in \mathbb Z_+^n, \; k_1 +k_2 +\dots + k_n \leq A/\min \{ \xi_j\}}
(k_1 + k_2 + \dots + k_n) \\
& \leq & \displaystyle
\Bxi \cdot \Bell + 2 \max \{ \xi_j \} C \left( \left[ \frac{A}{\min \{ \xi_j \}} \right] + 1 \right)^{n+1}.
\end{array}
\end{equation}
Hence $ | \hat{\mu} | ([-A,A])   $ grows at most polynomially ($ \sim A^{n+1}$)
and
therefore determines a tempered distribution.

To complete the proof we invoke  Theorem 11 of \cite{Fa19} which asserts that our translation bounded $ \mu $
which has countable spectrum is an almost periodic measure in the sense of \cite{Me16}*{Definition 5}.
\end{proof}

\subsection*{Remarks}
\begin{itemize}
\item  Starting with the function $ G $ instead of $ F $ we get a similar summation  formula
\begin{equation}
\sum_{\gamma: F(i \gamma) = 0}  \hat{h} (-\gamma) =
 (\Bxi \cdot \Bell) h(0)
- \sum_{\mathbf k \in \mathbb Z_+^n }
 (\Bxi \cdot \mathbf k) c_Q (\mathbf k) h  (\Bxi \cdot \mathbf k) -
  \sum_{\mathbf k \in \mathbb Z_+^n }
 (\Bxi \cdot \mathbf k) c_P (\mathbf k) h  (-\Bxi \cdot \mathbf k).
\end{equation}
Summing the two formulas we get
\begin{equation}
\sum_{\gamma: F(i \gamma) = 0}  \Big( \hat{h} (\gamma) + \hat{h} (-\gamma) \Big) =
2  (\Bxi \cdot \Bell) h(0)
- \sum_{\mathbf k \in \mathbb Z_+^n }
 (\Bxi \cdot \mathbf k)  \Big( c_P (\mathbf k) + c_Q (\mathbf k) \Big) \Big( h  (\Bxi \cdot \mathbf k) + h  (-\Bxi \cdot \mathbf k)
 \Big).
\end{equation}

\item In the self dual case $ P({\bf z} ) = Q({\bf z}) $ the summation formula takes the simplest form
\begin{equation} \label{sumsym}
\sum_{\gamma: F(i \gamma) = 0}  \hat{h} (\gamma) =
 (\Bxi \cdot \Bell) h(0)
- \sum_{\mathbf k \in \mathbb Z_+^n }
 (\Bxi \cdot \mathbf k) c_P (\mathbf k) \Big( h  (\Bxi \cdot \mathbf k) + h  (-\Bxi \cdot \mathbf k) \Big).
\end{equation}

\item The simplest stable polynomial is
$$ P(z_1) = 1-z_1 . $$
For it
$$ \begin{array}{ccl}
\displaystyle Q(z_1) & = & \displaystyle  z_1-1; \\
 \displaystyle F(s) & = &  \displaystyle 1- 1/b_1^s; \\
  \displaystyle \gamma_n & = &  \displaystyle \frac{2 \pi}{\xi_1} n, \quad n \in \mathbb Z; \\
 \displaystyle \log F(s) & = &  \displaystyle \log (1- \frac{1}{b_1^s})  = - \sum_{n=1}^\infty \frac{1}{n} \frac{1}{(b_1^n)^s}; \\
  \displaystyle \frac{F'}{F} (s) & = &  \displaystyle \sum_{n=1}^\infty \frac{1}{(b_1^n)^s} \xi_1; \\
 \end{array}
 $$
Substitution into the summation formula \eqref{sum2} gives
\begin{equation}
\sum_{n \in \mathbb Z} \hat{h} \big( \frac{2\pi}{\xi_1} n \big) =
\xi_1 \left(  h(0) + \sum_{n=1}^\infty  \Big( h (n \xi_1) +   h (- n \xi_1)  \Big)\right) \equiv \xi_1 \sum_{n \in \mathbb Z}  h (n \xi_1) ,
\end{equation}
which is nothing else than the classical Poisson summation formula (properly scaled) (see \eqref{Psf} below).

\end{itemize}

\section{The first non-trivial example}   \label{SexEx}

Our goal in this section is to present an explicit example of a positive crystalline measure.
 Consider the following polynomial
\begin{equation} \label{eq22}
P (z_1, z_2) = 1 - \frac{1}{3} z_1 + \frac{1}{3} z_2^2 - z_1 z_2^2 ,
\end{equation}
 in fact describing the non-linear part of the spectrum of the lasso graph \cite{KuSa2}.
With $ \ell_1 = 1 $, $ \ell_2 = 2 $ and $ \eta = - 1$ we get
$$ Q(z_1, z_2) =  (-1) \Big( z_1 z_2^2 - \frac{1}{3} z_2^2 + \frac{1}{3} z_1 - 1\Big) \equiv  P(z_1, z_2). $$
The polynomial is $ \mathbb D$-stable since the equation
$ P(z_1, z_2) = 0 $
can be writen as
$$ \frac{z_1 -3}{1-3 z_1} = z_2^2 $$
and the M\"obius transformation $ z_1 \mapsto \frac{z_1-3}{1-3 z_1}  $ maps  the unit disk to its complement.

The Dirichlet series is equal to
$$ F(s) = 1 - \frac{1}{3} \frac{1}{b_1^s} + \frac{1}{3} \frac{1}{b_2^{2s}} - \frac{1}{b_1^s b_2^{2s}} . $$
$$\begin{array}{ccl}
\displaystyle \log F(s) & = & \displaystyle - \sum_{k=1}^\infty  \frac{1}{k} \Big(\frac{1}{3} \frac{1}{b_1^s} - \frac{1}{3} \frac{1}{b_2^{2s}} + \frac{1}{b_1^s b_2^{2s}} \Big)^k  \\
& = & \displaystyle  \sum_{(n_1,n_2) \in \mathbb Z_+^2} c (n_1, 2 n_2) \frac{1}{b_1^{n_1 s} b_2^{2n_2s}},
\end{array}
$$
with
\begin{equation} \label{c_n}
c(n_1, 2 n_2) = - \sum_{\stackrel{\scriptstyle  k_1, k_2, k_3 \in \mathbb N \cup \{0\}}{\stackrel{\scriptstyle k_1 + k_3 = n_1}{\scriptstyle
k_2+k_3= n_2}}} \frac{(k_1+k_2 +k_3-1)!}{k_1!\;  k_2! \; k_3!} \frac{(-1)^{k_2}}{3^{k_1+k_2}}. 
\end{equation}
To determine the zero set of $ F(s) $ let us first describe the zero set of $ P $ on the unit torus $ \mathbb T = \left\{ (z_1, z_2) \in \mathbb C^2:
|z_1 | = |z_2| = 1 \right\}. $
Introducing notations $ z_1 = e^{i x}, z_2 = e^{iy} $ the same torus  can be seen as the square $ [0,2\pi] \times [0, 2 \pi] $
with the opposite sides identified.

Then the zero set is described by the Laurent polynomial
$$
 L (x, y)  = 3 \sin ( \frac{x}{2} + y) + \sin (\frac{x}{2} - y) 
 $$
and is  plotted in Figure \ref{Fig22}.
\begin{figure}[H]
\includegraphics[width=0.35\textwidth]{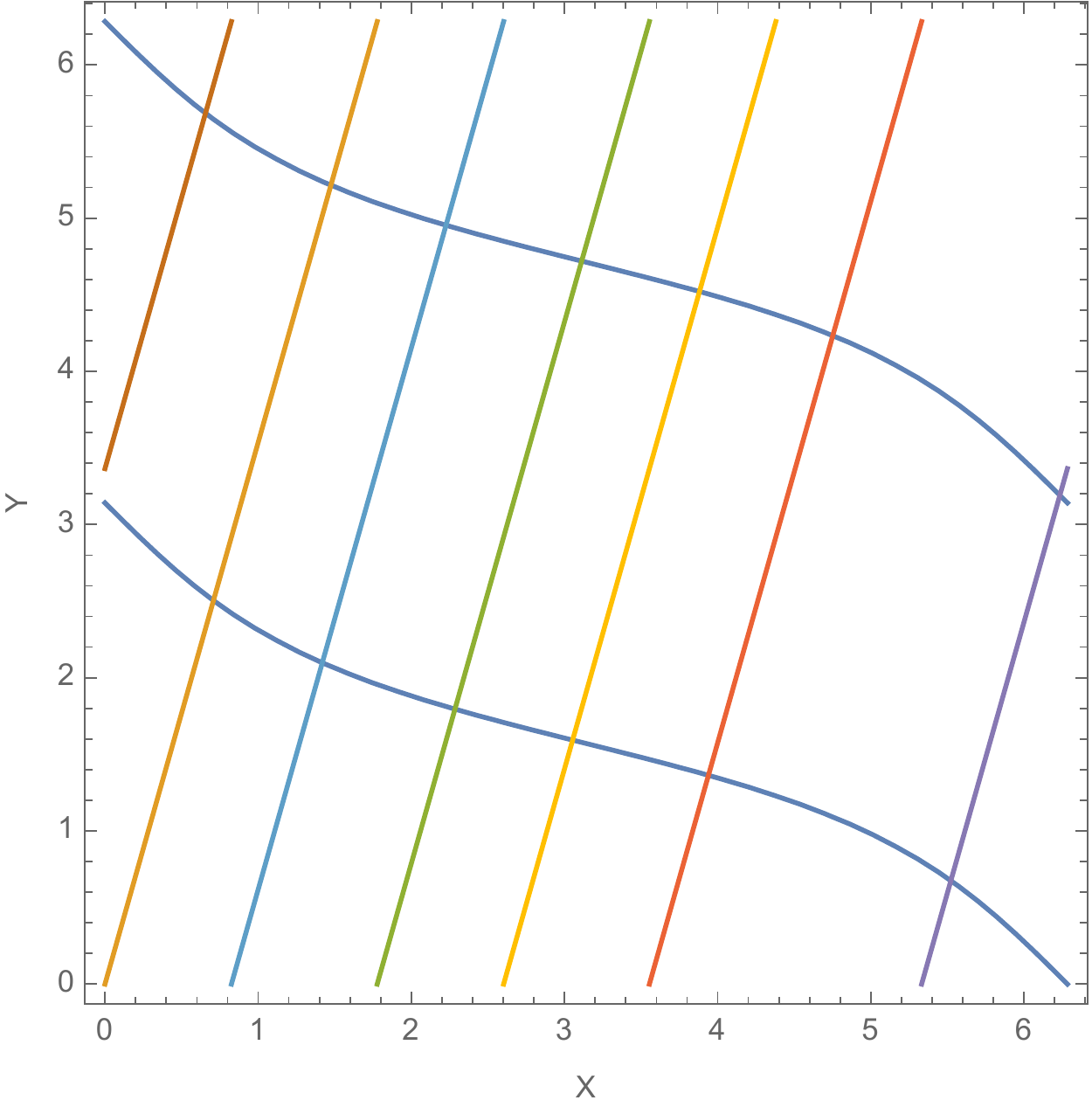}
\caption{Zero set for $ L (x, y)$.} \label{Fig22}
\end{figure}  
Note that the normal to the curve always  lies in the first quadrant, in fact
$$ 
\begin{array}{ccl}
\displaystyle \frac{\partial y}{\partial x} & = & \displaystyle - \frac{\frac{\partial L(x,y)}{\partial x}}{\frac{\partial L(x,y) }{\partial y}} \\
& = & \displaystyle - \frac{1}{2} \frac{3 \cos (x/2+y) + \cos (x/2-y)}{ 3 \cos (x/2+1) - \cos (x/2-y)} \\
& = & \displaystyle - \frac{1}{2} \frac{8}{(3 \cos (x/2+y) - \cos (x/2-y))^2} \\
& < & 0 ,
\end{array}$$
where we used that $ L(x,y) = 0. $

Knowing the zero set of $ L(x,y) $ the zeroes of the Dirichlet series $ F(s) $ (all lying on the imaginary axis) are obtained in the following way:
$$ 0 = F(i \gamma_j) = P ( b_1^{i\gamma_j}, b_2^{i \gamma_j}) = P ( e^{i \gamma_j \xi_1}, e^{i \gamma_j \xi_2}) =
L ( \gamma_j \xi_1, \gamma_j \xi_2) $$
\begin{equation} \label{sec}
\Leftrightarrow 
3 \sin \left(  (\frac{\xi_1}{2} + \xi_2) \gamma_j \right)+ \sin \left(  (\frac{\xi_1}{2} - \xi_2) \gamma_j \right)  = 0,
\end{equation}
where we used that $ \xi_j = \ln b_j > 0. $
In other words, zeroes of $ F $ are situated at the intersection points between the line  $ ( \gamma \xi_1, \gamma \xi_2)  $
and the zero curve for $ L. $  
Both the normal to the zero curve and the guide vector for the line belong to the first quadrant, hence the 
intersection is never tangential.
This implies in particular that all zeroes are simple.
$ \gamma_0 = 0 $ is always a solution since $ L(0,0) = 0. $ All other zeroes $ \gamma_j $ indicate the
distance between the intersection points and the origin measured along the line.
It is clear that $ L(-x,-y) = - L(x,y) $ (which also follows from \eqref{F-G} and the fact that  $ F = G $ in the
current example) implying that the zeroes are symmetric with respect to the origin.

The summation formula \eqref{sum2} takes the form
\begin{equation} \label{sum22}
\begin{array}{ccl}
\displaystyle \sum_{\gamma_j}  \hat{h} (\gamma_j) & = &
\displaystyle 
\left(\xi_1 + 2 \xi_2 \right) h(0) \\
& & \displaystyle 
-  \sum_{  \mathbf n = (n_1,n_2) \in \mathbb Z_+^2  }  
c(n_1, 2 n_2) (n_1 \xi_1 + 2 n_2 b_2) \Big(h (n_1 \xi_1 + 2 n_2  \xi_2) \\[-5mm]
& & \hspace{50mm} + h (-(n_1 \xi_1 + 2 n_2 \xi_2)) \Big),
\end{array}
\end{equation}
where
\begin{itemize}
\item $ \gamma_j $ are solutions to the secular equation \eqref{sec};
\item $ c(n_1, 2 n_2) $ are given by \eqref{sec};
\item $ h \in C_0^\infty (\mathbb R) $ - an arbitrary test function.
\end{itemize}
The difference between formula \eqref{sum22} and the general formula \eqref{sum2} is due to the fact that the stable polynomials depend just on $ z_2^2 $. 

Both series on the left and right hand sides are infinite but they have different properties depending on whether
$ \xi_1 $ and $ \xi_2 $ are rationally dependent or not. This is related to the number of intersection points on the
torus.
Also the number of zeroes $ i \gamma_j $ is always infinite, the number of intersection points
on the torus may be finite. Indeed, if $ \frac{\xi_1}{  \xi_2} \in \mathbb Q $, then
the line is periodic on the torus, implying that there are finitely many intersection points (on the torus).
 The points $ \gamma_j $ form a
periodic sequence implying that obtained  summation formula is just a finite sum of Poisson summation formulas
with the same period and $ \mu $ is a generalized Dirac comb.

Next we
assume that $ \xi_1 $ and $ \xi_2 $ are rationally independent
\begin{equation}
\frac{ \xi_1 }{ \xi_2} \notin \mathbb Q. 
 \end{equation}
By Kronecker's theorem the line covers the torus densely and therefore the intersection points $ (\gamma_j \xi_1, \gamma_j \xi_2) $
cover densely the zero curve of $ L $ as well. We are interested in the rational dependence  of
 $  \gamma_j, j \in \mathbb Z. $ In particular we shall need the following 
 
 \begin{lemma} \label{Le1}
 If $ \xi_1 $ and $ \xi_2 $ are rationally independent, then
 the secular equation \eqref{sec}
  $$ L (\gamma \xi_1, \gamma \xi_2) = 0 $$ has infinitely many
 rationally independent solutions, {\it i.e.}
 \begin{equation}
 \dim_{\mathbb Q} \mathcal L_{\mathbb Q} \{ \gamma_j \}_{j \in \mathbb Z} = \infty,
 \end{equation}
 where $ \mathcal L_{\mathbb Q} $ denotes the linear span with rational coefficients and
 $ \dim_{\mathbb Q} $ the dimension of the vector space with respect to the field $ \mathbb Q. $ 
 \end{lemma}
 \begin{proof}
 Assume that the dimension is finite. This means that there exists a certain $ M \in \mathbf N $ such that
 every $ \gamma_j $ for arbitrary $ j $ can be written as a rational combination of $ \gamma_1, \dots, \gamma_M $:
 \begin{equation}
 \gamma_j = a^j_1 \gamma_1 + a^j_2 \gamma_2 + \dots + a^j_M \gamma_M, \quad a^j_m \in \mathbb Q.
 \end{equation}
 It follows that
 $$ e^{i \gamma_j  \xi_\alpha} = \left( e^{i \gamma_1 \xi_\alpha} \right)^{a^j_1}    
 \left( e^{i \gamma_2 \xi_\alpha} \right)^{a^j_2}  \times \dots \times    \left( e^{i \gamma_M \xi_\alpha} \right)^{a^j_M}    , \quad \alpha = 1,2,
$$
or equivalently
\begin{equation}
b_\alpha^{i \gamma_j} = \left( b_\alpha^{ik_1} \right)^{a^j_1}   \left( b_\alpha^{ik_2} \right)^{a^j_2}  \times \dots \times  
 \left( b_\alpha^{ik_M} \right)^{a^j_M}  .
\end{equation}
Consider the multiplicative subgroup of $ (\mathbb C^*)^2 $ generated by
$$ (b_1^{i k_m}, b_2^{ik_m} ), \quad m =1,2, \dots, M $$
with the multiplication carried out coordinate wise. Then points $ (b_1^{ik_j}, b_2^{ik_j} ) $ belong to
the {\bf division group} $ \overline{\Gamma} $ of $ \Gamma$,
that is
$$ \overline{\Gamma} = \left\{ z \in (\mathbb C^*)^2: z^m \in \Gamma \; \mbox{for some $ m \geq 1$} \right\}.$$
In accordance with S. Lang's conjecture \cite{La65}  intersection between any algebraic subvariety and the division group for a finitely  
generated subgroup is along a finitely many subtori. The following theorem is proven in \cite{Li74}

\begin{theorem*}[Liardet \cite{Li74}] \label{ThLiardet}
Assume that:
\begin{itemize}
\item  $ \Gamma $ is a finitely generated subgroup of the multiplicative group of the complex torus $ (\mathbb C^*)^2; $
\item $ \overline{\Gamma} $ is the division group of $ \Gamma $;
\item $ V \subset (\mathbb C^*)^2 $ is an algebraic subvariety given by the zero set of Laurent polynomials.
\end{itemize}
Then the intersection of $ V $ and $ \overline{\Gamma} $ belongs to the union of a finitely many translates of 
certain subtori $ T_1, \dots, T_\nu $ contained in $ V $:
\begin{equation} \label{eq49}
V \cap \overline{\Gamma} = V \cap \Big( T_1 \cup T_2 \cup \dots \cup T_\nu \big).
\end{equation}
\end{theorem*}

Now no line belongs to the zero set of $ L $, so $ L $
contains no one dimensional subtori and  hence the intersection of the zero set (the curves plotted in Figure \ref{Fig22}) and 
the union of $ T_j $ in \eqref{eq49}
 is also finite. This contradicts the fact that the number of intersection points is infinite if $ \xi_1 $ and
$ \xi_2 $ are rationally independent, which completes the proof.
\end{proof}

Our main result can be formulated as

\begin{theorem} \label{Thmu}
For $ \xi_1 /\xi_2 \notin \mathbb Q $, the Fourier quasicrystal measure $ \mu $
corresponding to $ P $ in \eqref{eq22}, satisfies:
\begin{enumerate}[i)]
\item $ a_\lambda = 1 $ for $ \lambda \in \Lambda$, that is $ \mu $ is a positive  ``{idempotent}''.
\item $ \dim_{\mathbb Q} \Lambda = \infty, \quad \dim_{\mathbb Q} S = 2, $ in
particular $ \mu $ is not a generalized Dirac comb.
\item $ \Lambda $  meets any arithmetic progression in $ \mathbb R$ 
in a finite number of points.
\item $ \Lambda $ is a Delone set (that is the minimal distance between elements of $ \Lambda $ is bounded
below by a positive constant and $ \Lambda $ is relative dense in $ \mathbb R$) while $ S $ is
not a Delone set.
\item $ | \hat{\mu}| $ is not translation bounded.  
\end{enumerate}
\end{theorem}
\begin{proof}
That $ \mu $ is a Fourier quasicrystal follows from Theorem \ref{Th1}. Note however that
the argument with $ c(n_1, 2 n_2) $ being Fourier
coefficients for $ \log P $ on the torus is especially transparent, since $ P $ is real on $ \mathbb T^2$  and $ \log P $ has just logarithmic singularities on the smooth
curve $ L(x,y) = 0 $  and therefore is absolutely integrable.

{\it i)} All zeroes of the secular equation \eqref{sec} have multiplicity one and form a discrete set, hence  by 
construction $ a_\lambda = 1 $ and  $ \mu $ is a positive idempotent discrete 
measure.

{\it ii)}  Since  $ \xi_1 / \xi_2 \notin \mathbb Q $ Lemma  \ref{Le1}
implies that 
 $ \dim_{\mathbb Q} \Lambda_P  = \infty $, hence
the support of $ \mu $ is not contained in a finite union of translates of any lattice
and  $ \mu $ is not a generalized Dirac comb. The spectrum $ S_P $ -- the support of $ \hat{\mu} $ --
belongs to
$$ L_+ \cup - L_+ \cup \{ 0 \} $$
and its dimension is equal to $2.$  

{\it iii)}  Assume that there exists a full arithmetic progression, say $ \gamma^* n  $
which intersects support of $ \mu $ at an infinite number of points.  Consider the
corresponding group generated by  $ (b_1^{\gamma^* }, b_2^{\gamma^* }). $ Its intersection with
the algebraic subvariety $ P(z_1, z_2)  = 0 $ (where $ P $ is given by \eqref{eq22}) is along a finitely many subtori
(Liardet's Theorem) as before. The zero set contains no one-dimensional subtori, hence the number of intersection points
on the torus is finite. The number of intersection points between the arithmetic progression and the zero set
can be infinite only if certain points occur several times, but this is impossible since $ \xi_1 /\xi_2 \notin
\mathbb Q. $ It follows that the intersection between any arithmetic progression and $ \Lambda $ is always
finite. The same result could be proven using Lech's theorem (Lemma on page 417 in \cite{Le53}).

{\it iv)} The zero set of  $ L(x,y) $ is given by two non-intersecting curves on $ [0, 2 \pi]^2 $ implying 
that there is a minimal distance $ \rho $ between the different components of the curve. Taking into account that
the intersection between the line $ (\xi_1 \gamma, \xi_2 \gamma) $ and the
zero curve of $ L(x,y) $ is non-tangential we conclude that there is a minimal distance between the
consecutive solutions $ \gamma_j $ of the secular equation \eqref{sec}.
The function $ L (\gamma \xi_1 ,\gamma \xi_2 ) $
is given by a sum of two sinus functions with amplitudes $3$ and $1$ implying that every
interval $ [ n \frac{2 \pi}{\frac{\xi_1}{2} + \xi_2}, (n+1) \frac{2 \pi}{\frac{\xi_1}{2} + \xi_2} ] $
contains a solution to the secular equation. It follows that the support of $ \mu $ is relatively dense and uniformly 
discrete, {\it i.e.} is a Delone set. The spectrum $ S_P $ is not a Delone set, since otherwise the measure $ \mu $
would be a generalized Dirac comb \cite{LeOl15}. 

 {\it v)} Similarly $ | \hat{\mu} |$ is not translation bounded  since otherwise this would contradict Meyer's Theorem stating that
every crystalline measure with $ a_\lambda $ 
from a finite set   ($ a_\lambda =1 $ in our case) and $ |\hat{\mu}| $
translation bounded is a generalized Dirac comb
(see Introduction and \cite{Me70}).
\end{proof}

{\bf Remarks to Theorem \ref{Thmu}}:
\begin{itemize}
\item Properties (ii) and (iii) show that the measures $ \mu $ in the theorem are far from being generalized Dirac combs.
\item In Theorem 5.16 of \cite{Me18} a positive measure $ \mu $ of the type in \eqref{mudef} is constructed for which $ \Lambda $
is discrete but for which $ S $ need not be (called there a Poisson measure). In fact these $ \Lambda$'s can be realized as the intersection of the graph
of a periodic continuous function on the two torus with an irrational line, and as such are of a similar shape to our $ \mu$'s.
\end{itemize}
\vspace{5mm}

The measures $ \mu $ in Theorem \ref{Thmu} provide examples answering the following questions
concerning crystalline measures:
\begin{enumerate}[(A)]
\item The last question in \cite{Me16}):
\newline a positive crystalline measure which is not a generalized Dirac comb;
\item Part 3 of question 11.2 in \cite{LeOl17}:
\newline a positive Fourier quasicrystal for which every arithmetic progression meets the support in a finite set;
\item  The question on page 3158 of \cite{Me16} and part 2 of question 11.2 in \cite{LeOl17})
\newline
a Fourier quasicrystal for which the support (that is $\Lambda$) is a Delone set, but
the spectrum (that is $S$) is not;
\item Problem 4.4 in \cite{La00}:
\newline a discrete set (that is $\Lambda$) which is a Bohr almost periodic Delone set, but is not an ideal crystal.
\end{enumerate}

In our forthcoming paper \cite{KuSa2} we use higher dimensional quantitative theorems from Diophantine
analysis \cite{Ev99,EvSchSch02,La84,Sch99} to show that general crystalline $ \mu $ 
constructed in Section 2
using a stable pair $ P, Q $ 
with parameters $ b_1, \dots, b_n $, satisfies:

\begin{theorem} \label{ThLast}
For such a
$ \mu $ we have that
\begin{enumerate}[i)]
\item $ \Lambda = L_1 \sqcup L_2 \sqcup \dots \sqcup L_{\nu} \sqcup N, $
with
$ L_1, \dots, L_\nu $ full arithmetic progressions and $ N $ if not empty is infinite dimensional over $ \mathbb Q $ (the
union $ \sqcup $ means counted with multiplicities).
\item $ a_\lambda $ take values in a finite set of positive integers; $ \mu $ is a positive Fourier
quasicrystal.
\item $ \dim_{\mathbb Q} S = \dim_{\mathbb Q} \left\{ \xi_1, \dots, \xi_n \right\}. $
\item There is $ c = c(P) < \infty $ such that any arithmetic progression in $ \mathbb R_+ $
meets $ N $ in at most $ c(P) $ points.
\end{enumerate}
\end{theorem}

{\bf Remarks to Theorem \ref{ThLast}}:
\begin{itemize}
\item Theorem \ref{ThLast} allows us to make $ \mu$'s for which $ \dim_{\mathbb Q} S $ is as
large as we please, however in as much as any positive crystalline measure is (measure) almost periodic
it follows from Lemma 5 of \cite{Me16} that $ S \cap (0, \infty) $ or $ S \cap (-\infty,0) $ cannot be linearly
independent over $ \mathbb Q. $
\end{itemize}

\section{Acknowledgement}

The authors would like to thank Boris Shapiro for initiating our collaboration, Yves Meyer
for attracting our attention to crystalline measures and pointing us to his paper \cite{Me18}, 
 Alexei Poltoratskii for pointing out importance of
positive crystalline measures, and Nir Lev and Alexander Olevskii for their comments.


\begin{thebibliography}{99}

\bib{BaGa}{article}{
   author={Barra, F.},
   author={Gaspard, P.},
   title={On the level spacing distribution in quantum graphs},
   note={Dedicated to Gr\'{e}goire Nicolis on the occasion of his sixtieth
   birthday (Brussels, 1999)},
   journal={J. Statist. Phys.},
   volume={101},
   date={2000},
   number={1-2},
   pages={283--319},
   issn={0022-4715},
   review={\MR{1807548}},
   doi={10.1023/A:1026495012522},
}



\bib{BoBr09}{article}{
   author={Borcea, Julius},
   author={Br\"{a}nd\'{e}n, Petter},
   title={The Lee-Yang and P\'{o}lya-Schur programs. I. Linear operators
   preserving stability},
   journal={Invent. Math.},
   volume={177},
   date={2009},
   number={3},
   pages={541--569},
   issn={0020-9910},
   review={\MR{2534100}},
   doi={10.1007/s00222-009-0189-3},
}
	


	







\bib{CdV}{article}{
   author={Colin de Verdi\`ere, Yves},
   title={Semi-classical measures on quantum graphs and the Gau\ss  map of the
   determinant manifold},
   journal={Ann. Henri Poincar\'{e}},
   volume={16},
   date={2015},
   number={2},
   pages={347--364},
   issn={1424-0637},
   review={\MR{3302601}},
   doi={10.1007/s00023-014-0326-4},
}

\bib{Co88}{article}{
   author={C\'{o}rdoba, Antonio},
   title={La formule sommatoire de Poisson},
   language={French, with English summary},
   journal={C. R. Acad. Sci. Paris S\'{e}r. I Math.},
   volume={306},
   date={1988},
   number={8},
   pages={373--376},
   issn={0249-6291},
   review={\MR{934622}},
}
		





\bib{Dy09}{article}{
   author={Dyson, Freeman},
   title={Birds and frogs},
   journal={Notices Amer. Math. Soc.},
   volume={56},
   date={2009},
   number={2},
   pages={212--223},
   issn={0002-9920},
   review={\MR{2483565}},
}
	


\bib{Ev99}{article}{
   author={Evertse, Jan-Hendrik},
   title={Points on subvarieties of tori},
   conference={
      title={A panorama of number theory or the view from Baker's garden},
      address={Z\"{u}rich},
      date={1999},
   },
   book={
      publisher={Cambridge Univ. Press, Cambridge},
   },
   date={2002},
   pages={214--230},
   review={\MR{1975454}},
   doi={10.1017/CBO9780511542961.015},
}

\bib{EvSchSch02}{article}{
   author={Evertse, J.-H.},
   author={Schlickewei, H. P.},
   author={Schmidt, W. M.},
   title={Linear equations in variables which lie in a multiplicative group},
   journal={Ann. of Math. (2)},
   volume={155},
   date={2002},
   number={3},
   pages={807--836},
   issn={0003-486X},
   review={\MR{1923966}},
   doi={10.2307/3062133},
}

\bib{Fa19}{article}{
   author={Favorov, S. Yu.},
   title={Large Fourier quasicrystals and Wiener's theorem},
   journal={J. Fourier Anal. Appl.},
   volume={25},
   date={2019},
   number={2},
   pages={377--392},
   issn={1069-5869},
   review={\MR{3917950}},
   doi={10.1007/s00041-017-9576-0},
}

		

\bib{Gu59}{article}{
   author={Guinand, A. P.},
   title={Concordance and the harmonic analysis of sequences},
   journal={Acta Math.},
   volume={101},
   date={1959},
   pages={235--271},
   issn={0001-5962},
   review={\MR{107784}},
   doi={10.1007/BF02559556},
}
			
\bib{GuSm}{article}{
   author={Gutkin, Boris},
   author={Smilansky, Uzy},
   title={Can one hear the shape of a graph?},
   journal={J. Phys. A},
   volume={34},
   date={2001},
   number={31},
   pages={6061--6068},
   issn={0305-4470},
   review={\MR{1862642}},
   doi={10.1088/0305-4470/34/31/301},
}




\bib{Ko16}{article}{
   author={Kolountzakis, Mihail N.},
   title={Fourier pairs of discrete support with little structure},
   journal={J. Fourier Anal. Appl.},
   volume={22},
   date={2016},
   number={1},
   pages={1--5},
   issn={1069-5869},
   review={\MR{3448912}},
   doi={10.1007/s00041-015-9416-z},
}


\bib{KuJFA}{article}{
   author={Kurasov, Pavel},
   title={Schr\"{o}dinger operators on graphs and geometry. I. Essentially
   bounded potentials},
   journal={J. Funct. Anal.},
   volume={254},
   date={2008},
   number={4},
   pages={934--953},
   issn={0022-1236},
   review={\MR{2381199}},
   doi={10.1016/j.jfa.2007.11.007},
}
		

\bib{KuArk}{article}{
   author={Kurasov, Pavel},
   title={Graph Laplacians and topology},
   journal={Ark. Mat.},
   volume={46},
   date={2008},
   number={1},
   pages={95--111},
   issn={0004-2080},
   review={\MR{2379686}},
   doi={10.1007/s11512-007-0059-4},
}
		


\bib{KuSa2}{article}{
   author={Kurasov, Pavel},
   author={Sarnak, Peter},
   title={The additive structure of the spectrum of a quantum graph},
   journal={in preparation},
   volume={},
   date={2020},
   number={},
   pages={},
%   issn={0305-4470},
 %  review={\MR{2148632}},
 %  doi={10.1088/0305-4470/38/22/014},
}


\bib{La00}{article}{
   author={Lagarias, Jeffrey C.},
   title={Mathematical quasicrystals and the problem of diffraction},
   conference={
      title={Directions in mathematical quasicrystals},
   },
   book={
      series={CRM Monogr. Ser.},
      volume={13},
      publisher={Amer. Math. Soc., Providence, RI},
   },
   date={2000},
   pages={61--93},
   review={\MR{1798989}},
}

\bib{La65}{article}{
   author={Lang, Serge},
   title={Report on diophantine approximations},
   journal={Bull. Soc. Math. France},
   volume={93},
   date={1965},
   pages={177--192},
   issn={0037-9484},
   review={\MR{193064}},
}


\bib{La84}{article}{
   author={Laurent, Michel},
   title={\'{E}quations diophantiennes exponentielles},
   language={French},
   journal={Invent. Math.},
   volume={78},
   date={1984},
   number={2},
   pages={299--327},
   issn={0020-9910},
   review={\MR{767195}},
   doi={10.1007/BF01388597},
}

\bib{Le53}{article}{
   author={Lech, Christer},
   title={A note on recurring series},
   journal={Ark. Mat.},
   volume={2},
   date={1953},
   pages={417--421},
   issn={0004-2080},
   review={\MR{56634}},
   doi={10.1007/BF02590997},
}


\bib{YaLe2}{article}{
   author={Lee, T. D.},
   author={Yang, C. N.},
   title={Statistical theory of equations of state and phase transitions.
   II. Lattice gas and Ising model},
   journal={Phys. Rev. (2)},
   volume={87},
   date={1952},
   pages={410--419},
   issn={0031-899X},
   review={\MR{53029}},
}
		

\bib{LeOl15}{article}{
   author={Lev, Nir},
   author={Olevskii, Alexander},
   title={Quasicrystals and Poisson's summation formula},
   journal={Invent. Math.},
   volume={200},
   date={2015},
   number={2},
   pages={585--606},
   issn={0020-9910},
   review={\MR{3338010}},
   doi={10.1007/s00222-014-0542-z},
}
		
\bib{LeOl16}{article}{
   author={Lev, Nir},
   author={Olevskii, Alexander},
   title={Quasicrystals with discrete support and spectrum},
   journal={Rev. Mat. Iberoam.},
   volume={32},
   date={2016},
   number={4},
   pages={1341--1352},
   issn={0213-2230},
   review={\MR{3593527}},
   doi={10.4171/RMI/920},
}
		
\bib{LeOl17}{article}{
   author={Lev, Nir},
   author={Olevskii, Alexander},
   title={Fourier quasicrystals and discreteness of the diffraction
   spectrum},
   journal={Adv. Math.},
   volume={315},
   date={2017},
   pages={1--26},
   issn={0001-8708},
   review={\MR{3667579}},
   doi={10.1016/j.aim.2017.05.015},
}
		


\bib{Li74}{article}{
   author={Liardet, Pierre},
   title={Sur une conjecture de Serge Lang},
   language={French},
   conference={
      title={Journ\'{e}es Arithm\'{e}tiques de Bordeaux},
      address={Conf., Univ. Bordeaux, Bordeaux},
      date={1974},
   },
   book={
      publisher={Soc. Math. France, Paris},
   },
   date={1975},
   pages={187--210. Ast\'{e}risque, Nos. 24-25},
   review={\MR{0376688}},
}
	
	\bib{Me70}{book}{
   author={Meyer, Yves},
   title={Nombres de Pisot, nombres de Salem et analyse harmonique},
   language={French},
   series={Lecture Notes in Mathematics, Vol. 117},
   note={Cours Peccot donn\'{e} au Coll\`ege de France en avril-mai 1969},
   publisher={Springer-Verlag, Berlin-New York},
   date={1970},
   pages={63},
   review={\MR{0568288}},
}
	

\bib{Me16}{article}{
   author={Meyer, Yves F.},
   title={Measures with locally finite support and spectrum},
   journal={Proc. Natl. Acad. Sci. USA},
   volume={113},
   date={2016},
   number={12},
   pages={3152--3158},
   issn={0027-8424},
   review={\MR{3482845}},
   doi={10.1073/pnas.1600685113},
}

\bib{Me18}{article}{
   author={Meyer, Yves F.},
   title={Global and local estimates on trigonometric sums},
   journal={Trans. R. Norw. Soc. Sci. Lett.},
  % volume={113},
   date={2018},
   number={2},
   pages={1--25},
%   issn={0027-8424},
  % review={\MR{3482845}},
 %  doi={10.1073/pnas.1600685113},
}


		


\bib{RaVi19}{article}{
   author={Radchenko, Danylo},
   author={Viazovska, Maryna},
   title={Fourier interpolation on the real line},
   journal={Publ. Math. Inst. Hautes \'{E}tudes Sci.},
   volume={129},
   date={2019},
   pages={51--81},
   issn={0073-8301},
   review={\MR{3949027}},
   doi={10.1007/s10240-018-0101-z},
}
	

\bib{Ru}{book}{
   author={Ruelle, David},
   title={Thermodynamic formalism},
   series={Encyclopedia of Mathematics and its Applications},
   volume={5},
   note={The mathematical structures of classical equilibrium statistical
   mechanics;
   With a foreword by Giovanni Gallavotti and Gian-Carlo Rota},
   publisher={Addison-Wesley Publishing Co., Reading, Mass.},
   date={1978},
   pages={xix+183},
   isbn={0-201-13504-3},
   review={\MR{511655}},
}



\bib{Sch99}{article}{
   author={Schmidt, Wolfgang M.},
   title={The zero multiplicity of linear recurrence sequences},
   journal={Acta Math.},
   volume={182},
   date={1999},
   number={2},
   pages={243--282},
   issn={0001-5962},
   review={\MR{1710183}},
   doi={10.1007/BF02392575},
}



\bib{Wa11}{article}{
   author={Wagner, David G.},
   title={Multivariate stable polynomials: theory and applications},
   journal={Bull. Amer. Math. Soc. (N.S.)},
   volume={48},
   date={2011},
   number={1},
   pages={53--84},
   issn={0273-0979},
   review={\MR{2738906}},
   doi={10.1090/S0273-0979-2010-01321-5},
}


		

\end{thebibliography}
\end{document}